\documentclass[12pt]{amsart}

\usepackage{color,fancyhdr}
\usepackage[all]{xy}

\usepackage{amsmath,amsthm,mathrsfs,amsfonts,amssymb,times,latexsym,mathabx}

\usepackage[usenames,dvipsnames,svgnames,table]{xcolor}

\DeclareMathAlphabet{\curly}{U}{rsfs}{m}{n}  %% curly font

\newtheorem{theorem}{Theorem}
\newtheorem{lemma}{Lemma}[section]
\newtheorem{proposition}[lemma]{Proposition}
\theoremstyle{definition}

\newtheorem{corollary}[theorem]{Corollary}

\theoremstyle{remark}

\usepackage{xcolor}

\numberwithin{equation}{section}

\newcommand{\ZZ}{{\mathbb Z}}

\newcommand{\NN}{{\mathbb N}}

\newcommand{\cA}{\ensuremath{\mathcal{A}}}

\newcommand{\cD}{\ensuremath{\mathcal{D}}}
\newcommand{\cE}{\ensuremath{\mathcal{E}}}

\newcommand{\sA}{\ensuremath{\mathscr{A}}}
\newcommand{\sB}{\ensuremath{\mathscr{B}}}

\newcommand{\sG}{\ensuremath{\mathscr{G}}}

\newcommand{\sP}{\ensuremath{\mathscr{P}}}

\makeatletter
\renewcommand{\pmod}[1]{\allowbreak\mkern7mu({\operator@font mod}\,\,#1)}
\makeatother

\newcommand{\dalign}[1]{\[\begin{aligned} #1 \end{aligned}\]}

\newcommand{\be}{\begin{equation}}
\newcommand{\ee}{\end{equation}}

\renewcommand{\ssum}[1]{\sum_{\substack{#1}}}  %%% stacked sum
\renewcommand{\sprod}[1]{\prod_{\substack{#1}}}  %% stacked product

\renewcommand{\a}{\ensuremath{\alpha}}
\renewcommand{\b}{\ensuremath{\beta}}
\newcommand{\del}{\ensuremath{\delta}}
\newcommand{\eps}{\ensuremath{\varepsilon}}

\renewcommand{\le}{\leqslant}

\renewcommand{\ge}{\geqslant}

\newcommand{\order}{\asymp}      %% order of magnitude
\renewcommand{\(}{\left(}
\renewcommand{\)}{\right)}
\newcommand{\ds}{\displaystyle}
\newcommand{\pfrac}[2]{\left(\frac{#1}{#2}\right)}  %%% frac with paren

\newcommand{\bb}{\ensuremath{\mathbf{b}}}

\begin{document}

\title[Divisors of polynomials]
{The distribution of divisors of polynomials}

\author{Kevin Ford}
\address{Department of Mathematics, University of Illinois at Urbana-Champaign, 1409 West
Green Street, Urbana, Illinois 61801, USA}
\email{ford126@illinois.edu}

\author{Guoyou Qian}
%    Address of record for the research reported here
\address{Department of Mathematics, Sichuan University, Chengdu 610064, P.R. China}
\email{qiangy1230@163.com, guoyouqian@scu.edu.cn}

\thanks{The first author was supported by National Science Foundation Grants  DMS-1501982 and DMS-1802139. }
\thanks{The second author was supported by National Science
Foundation of China Grant \#11501387 and by
International Visiting program for Excellent Young Scholars of Sichuan University.  He thanks the Univerisity of Illinois for hosting his visit from August, 2017 to August, 2018.}
\date{\today}%

\begin{abstract}
Let $F(x)$ be an irreducible polynomial with integer coefficients
and degree at least 2. For $x\ge z\ge y\ge 2$, denote by $H_F(x, y, z)$
the number of integers $n\le x$ such that $F(n)$ has at least one divisor $d$ with $y<d\le z$.
We determine the order of magnitude of $H_F(x, y, z)$ 
uniformly for $y+y/\log^C y < z\le y^2$ and $y\le x^{1-\delta}$, showing that the
order is the same as the order of $H(x,y,z)$, the number of
positive integers $n\le x$ with a divisor in $(y,z]$.
Here $C$ is an arbitrarily large constant and $\delta>0$ is arbitrarily small.
\end{abstract}

\maketitle

%%%%%%%%%%%%%%%%%%%%%%%%%%%%%%%%%%%%%%%%%%%%
%
\section{\bf Introduction}
%
%%%%%%%%%%%%%%%%%%%%%%%%%%%%%%%%%%%%%%%%%%%%

Let $F(t) \in \ZZ[t]$ be an irreducible polynomial
of degree $g\ge 2$.
In this paper we study the size of $H_F(x,y,z)$, the number of 
positive integers $n\le x$ for which $F(n)$ has a divisor
in $(y,z]$.  The special case $F(t)=t$, counting integers $n\le x$ with a divisor in $(y,z]$, is classical and goes back to
early work of Besicovitch and Erd\H os in the 1930s.  
In 2008, the first author \cite{Ford08-1} determined the exact order
of growth of $H(x,y,z)$ for all $x,y,z$.  In particular, we have
\be\label{Hxy2y}
H(x,y,2y) \asymp \frac{x}{(\log y)^{\cE}(\log\log y)^{3/2}} \qquad (10\le y\le \sqrt{x}),
\ee
where 
\[
\cE = 1 - \frac{1+\log\log 2}{\log 2} = 0.086071332\ldots.
\]

The corresponding estimate for a linear polynomial $F$
follows from an argument identical to that in \cite{Ford08-1},
uniformly in the coefficients (see e.g., Proposition 2 in \cite{FKSY}).
The study of $H_F(x,y,z)$ for a general polynomial 
began in connection with the problem of 
bounding from below the largest prime factor of
$\prod_{n\le x} F(n)$.  This problem began
with work of Chebyshev (see Markov \cite{Mar})
for $F(t)=t^2+1$ and 
has  received a great deal of attention since.
For work on bounding the largest prime factor of
$\prod_{n\le x} F(n)$ for  specific polynomials $F$,
see the important papers of Ivanov \cite{Iva},
 Hooley \cite{Hoo67}, Hooley \cite{Hoo78},
Deshouillers and Iwaniec \cite{DI},
 Heath-Brown \cite{Heath-Brown},
Irving \cite{Irv},  Dartyge \cite{Dar}, 
la Bret\`{e}che \cite{Bre},
la Bret\`{e}che and Drappeau \cite{BD}
and Merikowski \cite{Merikowski}.
The first bound on the largest prime factor of
$\prod_{n\le x} F(n)$ for
general $F$ is due to Nagell \cite{Na21}, 
and was subsequently improved by Erd\H{o}s \cite{Er52P},
Erd\H{o}s and Schinzel \cite{ErS}, and most recently by Tenenbaum
\cite{T90-2}.
 Erd\H{o}s and Schinzel \cite{ErS} gave the explicit lower bound
\[
\max \bigg\{ p : p \Big| \prod_{n\le x} F(n) \bigg\} \gg x\exp\Big\{ \frac{\log x}{gx}H_F(x, \frac{x}{2}, x)\Big\},
\]
where $g$ is the degree of $F$ (this bound is also implicit
in Erd\H os \cite{Er52P}).
The best lower bounds for $H_F(x,x/2,x)$ are due
to Tenenbaum \cite{T90-2}, who showed that
\be\label{Ten}
H_F(x,x/2,x) \gg_F x/(\log x)^{\log 4-1+o(1)} \qquad (x\to \infty).
\ee
In \cite{T90-1}, Tenenbaum took up the problem of bounding
$H_F(x,y,z)$ for general $x,y,z$.  There are technical difficulties
that arise when $y\gg x$, and thus Tenenbaum restricted his
attention to the case $y\le x^{1-\delta}$ for some fixed $\delta>0$.
In this case he proved the following (we compare with
the size of $H(x,y,z)$, as the order is now known).

\medskip

\noindent{\bf Theorem T1.} {\it Let $\delta>0$ and $C>1$ be real. Then 
if $y_0$ is large enough, depending only on $\delta,C,F$,
and also $y_0\le y\le x^{1-\delta}$ and $y+y/(\log y)^C \le z \le 2y$, then
$$
H_F(x, y, z) = H(x,y,z) 
\exp\{O_{\delta,C,F}(\sqrt{\log\log y\log\log\log y})\}.
$$
}

In particular, combined with \eqref{Hxy2y} 
we see that $H_F(x, y, 2y)=x(\log y)^{-\cE+o(1)}$ uniformly for $y_0 \le y \le x^{1-\delta}$.
Tenenbaum's paper deals with arbitrary polynomials, irreducible
or reducible.  In order to remove various technical
issues that pertain to reducible polynomials, we
focus here on the irreducible case.
We record here only one of the more important estimates of 
Tenenbaum in the reducible case; see
(1.13) in \cite{T90-1}.

\medskip

\noindent{\bf Theorem T2.} {\it 
Let $F\in \ZZ[x]$ be a reducible polynomial which factors as
$F(x)=\prod_{j=1}^r F_j(x)^{\alpha_j}$, where $F_1,\ldots,F_r$ 
are distinct and irreducible.   Define $\tau=-1+\sum_j \log(\alpha_j+1)$.  For any $\delta>0$,
there is a constant $C$ so that uniformly for
$y+y/(\log y)^{\tau-\delta} \le z$ and $y\le x^{1-\delta}$
we have $H_F(x,y,z) \order x$, the implied constants 
depending on $F,\delta$.
}

\medskip

Let $\rho(d)$ be the number of solutions of $F(n)\equiv 0\pmod{d}$.
Heuristically, we expect $H_F(x,y,z)$ to behave like $H(x,y,z)$
since the average of $\rho(p)$ over primes $p$
 is 1.  Consequently, the distribution of
the prime factors of $F(n)$, over a randomly chosen $n\le x$, should be very close to the distribution of the prime factors of 
a randomly chosen $n\le x$.  We confirm this heuristic below.
In order to facilitate future applications, we state a lower bound for the number of $n\in (x/2,x]$ with $F(n)$ having a divisor in
$(y,z]$.

%%%%%%%%%%%%%%%%%%
%%  main theorem 
%%%%%%%%%%%%%%%%%%

\begin{theorem}\label{thm:main}
Let $F(t) \in \ZZ[t]$ be irreducible.
Let $\delta>0$ be an arbitrarily small positive constant,
and $C>1$ an arbitrarily large constant.
For some sufficiently large $y_0=y_0(F,\delta,C)$, we have
\[
H_F(x, y, z) \ll H(x,y,z) \ll H_F(x,y,z)-H_F(x/2,y,z)
\]
uniformly in the range $y_0 \le y\le x^{1-\delta}$ and 
$y+y/\log^C y \le z\le y^2$.
The constants implied by $\ll$ may  depend on $F,\delta,C$.
\end{theorem}

Combining Theorem \ref{thm:main} with \eqref{Hxy2y}, we see that

\begin{corollary}\label{cor2}
Let $F(t) \in \ZZ[t]$ be irreducible.
Fix $\delta>0$.  There is a constant $y_0=y_0(\delta,F)$ such that
uniformly for $y_0 \le y\le x^{1-\delta}$, we have
$$
H_F(x, y, 2y)\asymp \frac{x}{(\log y)^{\cE}(\log\log y)^{3/2}}
$$
\end{corollary}

According the above heuristic,
it is natural to conjecture that the conclusion 
of Corollary \ref{cor2} holds in
a larger range of $y$, perhaps $y\le x^{g-\delta}$.
In particular, taking $y=x/2$, we conjecture that
when $g\ge 2$, $H_F(x,x/2,x)$ has order
$\frac{x}{(\log y)^{\cE}(\log\log y)^{3/2}}$.
If true, this is a large improvement over Tenenbaum's
bound \eqref{Ten}.

To prove Theorem \ref{thm:main},
we develop a hybrid of the methods from
\cite{T90-1} and \cite{Ford08-1}.
The proof of the lower bound is accomplished in Section
\ref{sec:lower}, and Section \ref{sec:upper} contains
 the proof of the upper bound.
 A crucial device used in the upper bound in \cite{Ford08-1} is not available in the
 context of divisors of polynomials, and we must develop an alternative approach.

We note the formula
\begin{align*}
H_F(x, y, z)=
\sum_{k=1}^{\infty}(-1)^{k-1}
\sum_{y<d_1<\cdots<d_k\le z}\Big(\Big\lfloor \frac{x}{ [d_1,..., d_k]}\Big\rfloor\rho(\text{lcm}[d_1,..., d_k]) +O\big(\rho(\text{lcm}[d_1,..., d_k])\big)\Big),
\end{align*}
a consequence of inclusion-exclusion.
However, this has too many summands to be of any use in 
bounding $H_F(x,y,z)$ unless the interval $(y,z]$ is
very short.

%%%%%%%%%%%%%%%%%%%%%%%%%%%%%%%%%%%%%%%%%%%%%%%%%%%%%%%%%%%%%%%
%
\section{Preliminaries}
%
%%%%%%%%%%%%%%%%%%%%%%%%%%%%%%%%%%%%%%%%%%%%%%%%%%%%%%%%%%%%%%%

\subsection{Notation}
The symbols $p,q$ (with or without subscripts) always denote primes.
Constants implied by $O, \ll, \gg$ and $\asymp$ symbols
 depend on $F$, $\delta$ and $C$ in Theorem \ref{thm:main}.
Dependence on any other parameter will be indicated,
e.g. by a subscript.
The notation $f\asymp g$ means $f\ll g$ and $g\ll f$.

The symbol $p$, with or without subscripts, always denotes a 
prime.
Let $P^+(n)$ be the largest prime factor of $n$, and $P^-(n)$
be the smallest prime factor of $n$. Adopt the conventions $P^+(1)=0$ and $P^-(1)=\infty$.
For any $t \ge s \ge 1$, denote by $\sP(s,t)$
the set of squarefree positive integers composed only of prime factors $p\in (s,t]$.  In particular, $1\in \sP(s,t)$ for any $s,t$.
Let
$$
\tau(n; y, z)= \# \{ d|n : y<d\le z \}.
$$
Given an integer $n\ge 1$, we say $d|n^\infty$ if
every prime factor of $d$ divides $n$.
As noted earlier, we denote by $\rho(d)$
the number of solutions of the congruence
$$F(n)\equiv0{\pmod d}.$$
It follows from the Chinese remainder theorem that $\rho(n)$ is a multiplicative function of $n$. Let $D_F$ be the discriminant of $F(X)$. Then we have
(cf. Theorems 42, 52, 54 of Nagell \cite{Na64}\footnote{Nagell uses the term ``primitive'' to refer to a polynomial
with the greatest common divisor of its coefficients equal to 1.})
for any prime $p$ and positive integer $a$,
\be\label{rho-upper}
\rho(p^a)\le \begin{cases} g & \text{ if } p\nmid D_F, \\
gD_F^2 & \text{ if } p|D_F.
\end{cases}
\ee
We also associate with $F$ an Euler-like function
\begin{align}\label{phiF}
\varphi_F(n):=n\prod_{p|n}(1-\rho(p)/p).
\end{align}
In particular, we have $\varphi_F(n)\ne 0$ whenever $P^-(n)>gD_F^2$.

As in \cite{Ford08-1}, for a given pair $(y, z)$ with $4\le y<z$, we define $\eta, u, \beta, \xi$ by
\begin{align}\label{etabetaxi}
z=e^{\eta}y=y^{1+u}, \quad \eta=(\log y)^{-\beta}, \quad
\beta=\log 4-1+\frac{\xi}{\sqrt{\log\log y}}.
\end{align}
For $z\le ey$, we need the following function
\begin{align}\label{Gbeta}
G(\beta)=\begin{cases}
\beta,  &\text{if}\ \beta\ge \log 4-1,\\
\frac{1+\beta}{\log 2}\log\Big(\frac{1+\beta}{e\log 2}\Big)+1, &\text{if}\ 0\le \beta\le \log 4-1,
\end{cases}
\end{align}
as well as
\be\label{z0}
z=z_0(y):=y\exp\{(\log y)^{1-\log 4}\}\approx y+y/(\log y)^{\log 4-1}.
\ee
With this notation, given any $\delta>0$,
we have \cite[Theorem 1]{Ford08-1}, uniformly for $3\le y\le x^{1-\delta}$,
\[
\frac{H(x,y,z)}{x} \order
\begin{cases}
\log(z/y)=\eta & y+1 \le z \le
  z_0(y) \\  \\
\dfrac{\b}{\max(1,-\xi) (\log y)^{G(\b)}} & z_0(y) \le z \le 2y \\ \\ 
u^\del (\log \tfrac{2}{u})^{-3/2} & 2y \le z \le y^2 \\ \\
1 & z \ge y^{2}.
\end{cases}
\]
Our goal is to show the same bounds for $H_F(x,y,z)$

\subsection{Bakground lemmata.}
Our first result is a consequence of the Prime Ideal
Theorem with classical de la Vall\'{e}e Poussin error
term (see \cite{Lan27}, Satz 190).

\begin{lemma}\label{Mertenformula}
There are two positive constants
$c_1$ and $c_2$, which depend on $F$, such that
\[
\sum_{p\le x}\frac{\rho(p)}{p}=\log\log x+c_1+O(e^{-c_2\sqrt{\log x}}).
\]
\end{lemma}

In \cite{Er52d}, Erd\H{o}s showed that $\sum_{n=1}^x \rho(n)>cx$ for some constant $c$ when $x$ is sufficiently large.
This was sharpened by Fomenko \cite{Fomenko} and Kim \cite{Kim},
the sharpest known bounds (for large degree $g$) being the
result of L\"u \cite{Lu}.  We shall only require a very weak
version of the bound.

\begin{lemma}[{\cite[Theorems 1.1,1.2]{Lu}}]\label{sumonrho(d)}
For any $\varepsilon > 0$, we have
$$
\sum_{d\le x}\rho(d)=A_F x+O_\eps(x^{1-3/(g+6)+\varepsilon})
$$
where $A_F$ is a constant depending on $F$.
\end{lemma}

We  need a generalization of a bound from Tenenbaum \cite{T90-1}.

\begin{lemma}[Tenenbaum {\cite[Lemma 3.4]{T90-1}}] \label{T90-1-L34}
Suppose that $Q$ is an integer divisible by $D_F$.
Let $d_0$, $d_1$ be two integers such that
$d_0|Q^{\infty},\ (d_1, Q)=1$. Let $K, K'$ be real number satisfying $0<K<1\le K'$. Suppose that $x$ is sufficiently large, depending only on $K,K'$. Then there is a positive constant $c_3=c_3(K, F, Q)<1$ such that under the conditions
$$
2\le t\le x^{c_3}, \quad
d_0d_1\le x^{1-K}, \quad P^+(d_0d_1)\le t^{K'},
$$
we have
\begin{align}\label{Ten-eq35}
\ssum{n\le x, d_0d_1|F(n)\\ p|F(n)\Rightarrow p|Q d_1\ {\rm or}\ p>t}1\asymp_{K,K',Q} \frac{x}{\log t}\frac{\rho(d_0)}{d_0}\frac{\rho(d_1)}{\varphi_F(d_1)}.
\end{align}
Moreover, the same order lower bound follows when $n$ is
restricted to $(x/2,x]$.
In addition, the relation (\ref{Ten-eq35}) holds, replacing the sign $\asymp_{K,K'}$ with $\ll_{K,K',F}$, when $P^+(d_0d_1)> t^{K'}$
or $x^{c_3} < t\le x$.
\end{lemma}

\begin{proof}
This follows from the proof of Tenenabum {\cite[Lemma 3.4]{T90-1}}]; there, the lemma is proved when $Q=D_F F(1)$
and counting all $n\le x$,
but the same proof works for $x/2<n \le x$ and  an arbitrary $Q$ divisible by $D_F$.
Also, the case $x^{c_3} < t \le x$ is not
considered explicitly in \cite{T90-1}.
However, the stated result follows by applying
\cite[Lemma 3.4]{T90-1} with $t$ replaced by 
$t':=\min(x^{c_3},t)$, and noting that $\log t' \asymp_{K,F,Q} \log t$ when $t\le x$.
\end{proof}

\begin{lemma}[Tenenbaum {\cite[Lemma 3.7]{T90-1}}]\label{Ten-Lem37}
We have uniformly for $w\ge v\ge 2$, $x\ge 2$ that
$$
\Bigg| \bigg\{n\le x: \sprod{p^a \| F(n)\\ p\le v} p^a>w\bigg\}\Bigg|\ll
x\exp\bigg\{-c_4\frac{\log w}{\log v}\bigg\},
$$
where $c_4=c_4(F)$ is a positive constant.
\end{lemma}

\begin{lemma}[Norton {\cite[\S 4]{Nor76}}]\label{Norton}
Suppose $0\le h<m\le x$ and $m-h\ge \sqrt{x}$. Then
$$
\sum_{h\le k\le m}\frac{x^k}{k!}\asymp \min\Big(\sqrt{x}, \frac{x}{x-m}\Big)\frac{x^m}{m!}.
$$
\end{lemma}

To understand the global distribution of the divisors of integers, we introduce a function which measures the degree of clustering of the divisors of an integer $a$. For $\sigma>0$, we define
$$
\mathscr{L}(a; \sigma)=\{x\in \mathbb{R}: \tau(a; e^x, e^{x+\sigma})\ge 1\}
$$
and
$$
L(a; \sigma)={\rm meas}\mathscr{L}(a; \sigma),$$
where ${\rm meas}(\cdot)$ denotes Lebesgue measure.
We record easy bounds for $L(a;\sigma)$.

\begin{lemma}\label{Ford-lemma31}
We have

{\rm (i)} If $(a, b)=1$, then $L(ab; \sigma)\le \tau(b)L(a; \sigma)$;

{\rm (ii)} If $p_1<\cdots< p_k$, then
$$
L(p_1\cdots p_k; \sigma)\le \min_{0\le j\le k} 2^{k-j}(\log(p_1\cdots p_j)+\sigma);
$$

{\rm (iii)} For any $a\in \NN$ and $\sigma>0$ we have
\[
L(a;\sigma) \ge \sigma (2\tau(a) - W(a;\sigma)),
\]
where 
\be\label{Waeta}
W(a;\sigma)=| \{d|a,d'|a : |\log(d/d')| \le \sigma \}|.
\ee
\end{lemma}

\begin{proof}
Parts (i) and (ii) are proved in Ford {\cite{Ford08-1}, Lemma 3.1}.
To show part (iii), let $\cD$ be the set of divisors $d|a$ such that
there is no divisor $d'|a$ with $|\log(d/d')|\le \sigma$ (isolated 
divisors).  The desired inequality follows from the fact that
$L(a;\sigma) \ge \sigma |\cD|$ and
\[
W(a;\sigma) \ge \tau(a)+(\tau(a)-|\cD|). \qedhere
\]
\end{proof}

\begin{lemma}\label{sum-st}
For any $r<s<t$ and $\eta>0$ we have
\[
\sum_{a\in \sP(r,t)} \frac{L(a;\eta)\rho(a)}{a} \ll 
\pfrac{\log t}{\log s}^2 \sum_{a\in \sP(r,s)} \frac{L(a;\eta)\rho(a)}{a}.
\]
\end{lemma}

\begin{proof}
For any $a\in \sP(r,t)$, decompose $a$ uniquely as 
$a=a'a''$ where $P^+(a')\le s < P^-(a'')$,
and write $L(a;\eta) \le \tau(a'') L(a';\eta)$ from
Lemma \ref{Ford-lemma31} (i).
Using Lemma \ref{Mertenformula} we have
\[
\sum_{a''\in \sP(s,t)} \frac{\tau(a'')\rho(a'')}{a''}=
\prod_{s<p\le t} \(1 + \frac{2\rho(p)}{p}\) \ll
\pfrac{\log t}{\log s}^2
\]
and the proof is complete.
\end{proof}

%%%%%%%%%%%%%%%%%%%%%%%%%%%%%%%%%%%%%%%%%%%%%5
%
%
\section{Lower bound}\label{sec:lower}
%
%
%%%%%%%%%%%%%%%%%%%%%%%%%%%%%%%%%%%%%%%%%%%%%%%%%

In this section we prove the lower bound in Theorem \ref{thm:main}. As in \cite{Ford08-1}, we first bound $H_F(x, y, z)$ in terms of an average of $L(a; \eta)\rho(a)/a$.   This  be thought of as a kind of local-to-global principle.
Recall the definition \eqref{etabetaxi} of $\eta$.
 Also define 
\be\label{Q}
D = 10gD_F^2, \qquad Q = \prod_{p\le D} p.
\ee
By \eqref{rho-upper}, we have
\be\label{Q-phi}
\varphi_F(p^a) \ge \frac{p^a}{2} \qquad (p\nmid Q).
\ee

\begin{proposition}
\label{FQ-HLP}
Let $C$ and $\delta$ be two positive real numbers with $0<\delta<1$.
Suppose that $y$ is sufficiently large (depending on $F,\delta,C$),
$y<z=e^{\eta}y\le x^{1-\delta/2}$, and
$\frac{1}{\log^C y}\le \eta\le \log y$ (in particular, $z\le y^2$).  Then
\be\label{HF-lwr}
H_F(x, y, z)-H_F(x/2,y,z) \gg \frac{x}{\log^2 y}
\sum_{a\in \sP(D,z)} \frac{L(a; \eta)\rho(a)}{a}.
\ee
\end{proposition}

\begin{proof}
Define
\[
\nu = \min\bigg\{c_3\bigg(\frac{\delta}{3}, F, Q\bigg), \frac{\delta}{3}\bigg\}, \qquad \eps = \frac{\nu}{6g},
\]
with $c_3(\frac{\delta}{3}, F, Q)$ being defined as in Lemma \ref{T90-1-L34}.
Let
$$
\mathcal{A}=\{a\in \mathbb{N}: a\le y^{\nu}, \rho(a)>0, \mu^2(a)=1, (a, Q)=1\}.
$$
For $a\in \mathcal{A}$, we consider integers $n\in (x/2,x]$ such that
$F(n)$ has the decomposition
\begin{align}\label{decom-apb}
F(n)=apb
\end{align}
satisfying the following conditions
\begin{align*}
(*){\left\{
\begin{array}{rl}
&\text{(i)}\ p\ \text{is a prime factor of}\ F(n)\ \text{with}\ p>D\ \text{and}\ \log(\frac{y}{p})\in \mathscr{L}(a; \eta),\\
&{\rm (ii)}\ \text{Every prime factor } q|b \text{ satisfies }
q|apQ \text{ or } q>R:=\min(z,x^\nu).
\end{array}\right\}}
\end{align*}
If $F(n)$ satisfies
$(*)$, then there is a divisor $d$ of $a$ such
that $y<pd\le z$, which implies that 
\be\label{ap-bounds}
d\le a\le y^{\nu}<y^{1-\nu}\le \frac{y}{d}<p\le z.
\ee
In particular, $p>a$ implies that $(a, p)=1$. 
 
 With $n$ fixed, let $r(n)$ be the number of triples $a,p,b$ such that \eqref{decom-apb} holds subject to $(*)$.
 We  assume that $y$ is large enough so that $y^\nu > Q^2$.
 Thus, \eqref{ap-bounds} imply that $p>D$.
We claim that
$r(n)\ll 1$ for all $n$.
If $z \le x^{\nu}$, then $R=z$ and it is clear 
from \eqref{ap-bounds} that $a,p,b$ are unique.
Hence $r(n) \le 1$. 
If $z > x^{\nu}$, then $R=x^\nu$ and 
$y \ge z^{1/2} \ge x^{\nu/2}$.
Since $F(n) \ll x^g$ for $n\le x$, we see that $F(n)$
has $O(1)$ prime factors (counted with multiplicity) larger than $y^{\nu}$.
By \eqref{ap-bounds}, $a$ must contain all of the prime factors
of $F(n)$ which are below $y^\nu$, except for those dividing $Q$.
There are $O(1)$ possible ways of distributing the
prime factors of $F(n)$ which are $>y^{\nu}$ among the
numbers $b$ and $p$, and therefore $r(n) \ll 1$ in this case.

Therefore, we have
\be\label{HF-lwr-1}
\begin{split}
H_F(x,y,z)-H_F(x/2,y,z) &\ge \ssum{x/2<n\le x \\ r(n)>0} 1 
\gg \sum_{x/2<n\le x}r(n)\\
&\ge \sum_{a\in \cA} \;\; \sum_{\log(y/p)\in \mathscr{L}(a; \eta)}\ssum{x/2<n\le x, ap|F(n)\\
q|F(n) \ \Rightarrow \ q|apQ\text{ or } q>R}1.
\end{split}
\ee
Moreover, 
$$
ap\le az\le z^{1+\nu}\le x^{(1-\delta/2)(1+\nu)}\le x^{(1-\delta/2)(1+\delta/3)}<x^{1-\delta/6},
$$
as well as
$$
R\le x^{c_3(\frac{\delta}{3}, F,Q)}, \qquad P^+(pa)\le z\le R^{1/\nu}.
$$
Applying Lemma \ref{T90-1-L34} with $d_0=1$, $d_1=ap$,
$K=\delta/3$ and $K'=1/\nu$, we find that
$$
\ssum{x/2<n\le x, ap|F(n)\\ q|F(n)\Rightarrow q|apQ\ \text{or}\ q>R}1
\gg \frac{x}{\log R} \, \frac{\rho(ap)}{\varphi_F(ap)}
\gg\frac{x}{\log R}\frac{\rho(a)}{a}\frac{\rho(p)}{p}.
$$
Thus, from \eqref{HF-lwr-1}, we derive that
\begin{align*}
H_F(x, y, z)-H_F(x/2,y,z)\gg \frac{x}{\log y}\sum_{a\in \mathcal{A}}\frac{\rho(a)}{a}\sum_{\log(\frac{y}{p})\in \mathscr{L}(a; \eta)}\frac{\rho(p)}{p}.
\end{align*}
Now $\mathscr{L}(a; \eta)$ is the disjoint union of intervals of length between $\eta/2$ and $\eta$, and
$\eta \gg 1/(\log y)^C$ by assumption.  Hence, using $p>y^{1-\nu}$,  repeated application of  Lemma \ref{Mertenformula} implies
$$\sum_{\log(y/p)\in \mathscr{L}(a; \eta)}\frac{\rho(p)}{p}\gg \frac{L(a; \eta)}{\log y}.$$
We conclude that
\be\label{HF-lwr-2}
H_F(x, y, z)-H_F(x/2,y,z) \gg \frac{x}{\log^2 y}\sum_{a\in \mathcal{A}}
\frac{L(a;\eta)\rho(a)}{a}.
\ee

We next relax the condition $a\le y^{\nu}$ in the
summation over $a$.   Recall that $\eps = \nu/(6g)$.
We have
\be\label{sum-La-1}
\ssum{a\le y^{\nu}\\ (a, Q)=1\\  \mu^2(a)=1}\frac{L(a; \eta)\rho(a)}{a} \ge 
\ssum{a\in \sP(D,y^{\eps})}\frac{L(a; \eta)\rho(a)}{a} \(1 - \frac{\log a}{\log(y^{\nu})}\).
\ee
Write $\log a = \sum_{p|a} \log p$, $a=pf$ with
$(p,f)=1$, use
$\rho(fp)=\rho(p)\rho(f)\le g\rho(f)$ by \eqref{rho-upper}
and $L(pf;\eta)\le 2L(f;\eta)$ from Lemma \ref{Ford-lemma31} (i).
This gives
\dalign{
\ssum{a\in \sP(D,y^\eps)}\frac{L(a; \eta)\rho(a)\log a}{a} &\le 
2g \sum_{D<p\le y^{\eps}} \frac{\log p}{p} \ssum{f\in\sP(D,y^\eps)} \frac{L(f;\eta)\rho(f)}{f} \\
&\le 2g(\log(y^{\eps})+O(1)) \ssum{f\in\sP(D,y^\eps)} \frac{L(f;\eta)\rho(f)}{f},
}
by Mertens' estimate.  If $y$ is sufficiently large in terms of
$\eps$ and $F$, then   $2g(\log(y^{\eps})+O(1)) \le \frac{\nu}{2}\log y$.  Inserting this last bound into \eqref{sum-La-1}, we obtain
\[
\sum_{a\in \cA} \frac{L(a;\eta)\rho(a)}{a} = 
\ssum{a\le y^{\nu}\\ (a, Q)=1\\  \mu^2(a)=1}\frac{L(a; \eta)\rho(a)}{a} \ge \frac12 \sum_{a\in \sP(D,y^\eps)}
\frac{L(a; \eta)\rho(a)}{a}.
\]
Inserting this into \eqref{HF-lwr-2},
and applying Lemma \ref{sum-st} with $t=z$ and $s=y^\eps$,
 we conclude the proof.
\end{proof}

Next, as in \cite{Ford08-1}, we relate the sum over $a$ in Lemma \ref{FQ-HLP} to an average 
of the function $W(a;\eta)$ from \eqref{Waeta}.

\begin{lemma}\label{FQlower-W}
Let $C$ and $\delta$ be two positive real numbers with $0<\delta<1$.
Suppose $y$ is sufficiently large (depending on $F,C,\delta$),
$y<z=e^{\eta}y\le x^{1-\delta/2}$, and $\frac{1}{\log^C y}\le \eta\le \log y$ (in particular, $z\le y^2$). Then
$$
H_F(x, y, z) -H_F(x/2,y,z) \gg \frac{\eta(1+\eta)x}{\log^2 y}
\sum_{a\in \sP(\max(D,z/y),z)}\frac{(2\tau(a)-W(a;\eta))\rho(a)}{a}.
$$
\end{lemma}

\begin{proof}

In the summation on the right side of \eqref{HF-lwr}, 
decompose $a$ uniquely as $a=a'a''$, where
$P^+(a') \le z/y < P^-(a'')$.  As in \cite[Lemma 4.2]{Ford08-1},
the prime factors $\le z/y$ have little effect, and we do not lose much
using the trivial inequality
\[
L(a;\eta) \ge L(a'';\eta).
\]
Therefore, because $\rho$ is multiplicative,
\[
\ssum{a\in\sP(D,y^{\eps})}\frac{L(a; \eta)\rho(a)}{a}
\ge \ssum{a'\in \sP(D,z/y)} \frac{\rho(a')}{a'}
\ssum{a''\in \sP(\max(D,z/y),z)} 
\frac{L(a'';\eta)\rho(a'')}{a''}.
\]
Writing the sum on $a'$ as an Euler product, and then
using Lemma \ref{Mertenformula} we see that
\dalign{
 \ssum{a'\in \sP(D,z/y)} 
\frac{\rho(a')}{a'} = \sprod{D<p\le z/y} \(1+\frac{\rho(p)}{p}\) &\gg \exp\Bigg\{ \ssum{D<p\le z/y} \frac{\rho(p)}{p} \Bigg\} \gg 1 + \log(z/y) = 1+\eta.
}
Finally, in the sum over $a''$, we invoke 
Lemma \ref{Ford-lemma31} (iii) to obtain
$L(a'';\eta) \ge \eta(2\tau(a'')-W(\a'';\eta))$,
and the proof is complete.
\end{proof}

From Lemma \ref{FQlower-W}, to obtain a lower bound for $H_F(x, y, z)$, we need to provide an upper bound on the sum over $\frac{W(a; \eta)\rho(a)}{a}$. For the purpose, we partition the primes into sets $E_1, E_2,...$ and then consider those integers $a$ with a prescribed number of prime factors in each interval $E_j$.
The partition is similar to that in \cite[Section 4]{Ford08-1}. 
Each $E_j$ consists of the primes in an interval
$(\lambda_{j-1}, \lambda_j]$, where
$\lambda_0=D$ and $\lambda_j \approx \lambda_{j-1}^2$;
specifically, $\lambda_j$ is defined inductively for $j\ge 1$ as the largest prime so that
\begin{align}\label{partion-merten}
\sum_{\lambda_{j-1}<p\le \lambda_j}\frac{\rho(p)}{p}\le \log 2.
\end{align}
Note that $\rho(p)/p\le \log 2$ always holds when $p>\lambda_0=D$, so that each set $E_j$ is nonempty.

By Lemma \ref{Mertenformula}, we have
$$\log\log\lambda_j-\log\log\lambda_{j-1}=\log 2+O(e^{-c_2\sqrt{\log \lambda_{j-1}}}),$$
By summing the above sum from $r=1$ to $j$, we get $$\log\lambda_j-\log(D)=j\log2+O\Big(\sum_{r=1}^j e^{-c_2\sqrt{\log\lambda_{r-1}}}\Big)=j \log 2+O(1),$$
which implies that
\begin{align}\label{eq-rangeoflambda}
2^{j-c_5}\le \log\lambda_j\le 2^{j+c_5}\quad (j\ge 0)
\end{align}
for some absolute constant $c_5$.
For a vector ${\bf b}=(b_1,...,b_J)$ of non-negative integers,
let $\mathcal{A}(b)$ be the set of square-free integers $a$ composed of exactly $b_j$ prime factors from $E_j$ for each $j$.
The following is analogous to \cite[Lemma 4.7]{Ford08-1}.
Here $M$ is a sufficiently large constant, which depends only
on $F$, and hence $M$  depend on $c_1,c_2,c_4,c_5$ as well.

\begin{lemma}\label{lemWupperbound}
Suppose $\eta>0$, ${\bf b}=(b_1,..., b_h)$ and define $m=\min\{j: b_j\ge 1\}$. If $\eta<1$, further assume that
$m \ge M$ and that $b_j\le 2^{j/2}$ for each $j$. Then
\[
\sum_{a\in \mathcal{A}({\bf b})}\frac{W(a; \eta)\rho(a)}{a}\le \frac{(2\log2)^{b_m+\cdots+b_h}}{b_m!\cdots b_h!} \Big[1.01+
(2^{c_5} c_6 g)\eta\sum_{j=m}^h 2^{-j+b_m+\cdots b_j}\Big)\Big].
\qedhere
\]
for some absolute constant $c_6>0$.
\end{lemma}

\begin{proof}
Let $k=b_m+\cdots +b_h$. For $j\ge 0$, let $k_j=\sum_{i\le j}b_i$. Let $a=p_1\cdots p_k$, where
\begin{align}\label{prime-interval}
p_{k_{j-1}+1},\cdots, p_{k_j}\in E_j\quad (m\le j\le h)
\end{align}
and the primes in each interval $E_j$ are unordered. Since $W(p_1\cdots p_k; \eta)$ is the number of pairs $Y, Z\subseteq \{1, \cdots, k\}$ with
\begin{align}\label{YZdifference}
-\eta\le \sum_{i\in Y}\log p_i-\sum_{i\in Z}\log p_i
\le \eta,
\end{align}
we have
\begin{align}\label{W-upperbound}
\sum_{a\in \mathcal{A}({\bf b})}\frac{W(a; \eta)\rho(a)}{a}
\le \frac{1}{b_m!\cdots b_h!}\sum_{Y, Z\subseteq \{1, ..., k\}}
\ssum{p_1,..., p_k\\ (\ref{prime-interval}), (\ref{YZdifference})}
\frac{\rho(p_1\cdots p_k)}{p_1\cdots p_k}.
\end{align}
There are $2^k$ pairs $(Y,Z)$ with $Y=Z$, and thus these pairs contribute at most $\frac{(2\log 2)^k}{b_m!\cdots b_h!}$ to the right side of (\ref{W-upperbound}).
 
When $Y\ne Z$, let $I=\max(Y\Delta Z)$
and we will split off the term $\log p_I$ from the inequalities
\eqref{YZdifference}.
Define $E(I)$ by $k_{E(I)-1}<I\le k_{E(I)}$, so that $p_I\in E_{E(I)}$. Let
\be\label{ell-def}
\ell=\min\{j: \lambda_j\ge \eta^{-2}\}.
\ee
We distinguish two cases: (i) $E(I)>\ell$; (ii) $m\le E(I)\le \ell$.

Consider first a pair $Y,Z$ in case (i).
With $p_i$ all fixed for $i\ne I$, \eqref{YZdifference}
implies that $p_I$ lies in an interval of the form
$(U,U e^{2\eta}]$, where $U\ge \lambda_{E(I)}\ge \eta^{-2}$ depends on $p_i$ for $i\ne I$.  By \eqref{rho-upper},
\[
\sum_{U<p_I \le e^{2\eta} U} \frac{\rho(p_I)}{p_I} 
\le g \sum_{U<p_I \le e^{2\eta} U} \frac{1}{p_I} 
\le c_6 g \frac{\eta}{\log U} \le 
c_6 g \eta 2^{-E(I)+c_5}
\]
for an absolute constant $c_6$ (here we use the
Brun-Titchmarsh inequality for $\eta<1$ and the
Mertens' bound for primes when $\eta\ge 1$).
Therefore, with $Y$ and $Z$ fixed, the sum over $p_1,\ldots,p_k$
on the right side of \eqref{W-upperbound} is at most
\[
c_6 g \eta 2^{-E(I)+c_5} (\log 2)^{k-1},
\]
using \eqref{partion-merten}.  With $I$ fixed
there are $2^{k-1+I}$ pairs $Y,Z$.  We also have
\dalign{
\sum_{I=1}^k 2^{I-E(I)} = \sum_{j=m}^h 2^{-j} \sum_{k_{j-1}<I\le k_j} 2^I \le 2 \sum_{j=m}^h 2^{-j+k_j}.
}
 We find that the contribution to the right side of  (\ref{W-upperbound}) from those $Y, Z$ counted in case (i) is
\[
\le (2^{1+c_5}c_6 g)  \frac{\eta(2\log 2)^k}{b_m!\cdots b_h!}\sum_{j=m}^h 2^{-j+b_m+\cdots b_j}.
\]

In case (ii), \eqref{ell-def} implies
\be\label{eta-caseii}
\eta \le \lambda_{\ell-1}^{-1/2} \le \exp\{ -2^{\ell-1-c_5} \}.
\ee
Write
$$
a=a'p_{k_{\ell}+1}\cdots p_k, \quad a'=p_1\cdots p_{k_{\ell}}.
$$
By hypothesis, $Y\bigcap \{k_{\ell}+1,..., k\}=Z\bigcap\{k_{\ell}+1, ..., k\}$.  We use a trivial bound \eqref{partion-merten} for
the sums over $p_{k_\ell+1},\ldots,p_{k}$ on
 the right side of \eqref{W-upperbound},
summing over the $2^{k-k_\ell}$ possibilities for the set $Y\bigcap \{k_{\ell}+1,..., k\}=Z\bigcap\{k_{\ell}+1, ..., k\}$, then expressing the remaining sum over
$p_1,\ldots,p_{k_\ell}$, $Y\cap \{1,\ldots,k_\ell\}$ and 
$Z\cap \{1,\ldots,k_\ell\}$ in terms of a sum on $a'$.
We conclude that the contribution to the right side of
\eqref{W-upperbound}  
  from those $Y, Z$ counted in case (ii) is
\be\label{case-ii}
\le \frac{(2\log 2)^{k-k_\ell}}{b_{\ell+1}!\cdots b_h!}
\sum_{a'}\frac{(W(a'; \eta)-\tau(a'))\rho(a')}{a'}.
\ee
The factor $W(a';\eta)-\tau(a')$ arises due to our counting
only of sets with $Y\ne Z$.
From \eqref{Waeta}, we see that
\[
W(a';\eta)-\tau(a')=2\#\{(d_1,d_2) : d_1|a',d_2|a',  1 < d_2/d_1 \le e^\eta\}.
\]
Suppose $d_1|a', d_2|a'$ and $1<d_2/d_1\le e^{\eta}$. Let $d=(d_1, d_2), d_1=f_1d, d_2=f_2 d$ and $a'=df_1f_2a''$. 
Since $\rho(f_2) \le g^{\omega(f_2)} \le g^{k_\ell}$ by 
\eqref{rho-upper},
we obtain
\begin{align*}
\sum_{a'}\frac{(W(a'; \eta)-\tau(a'))\rho(a')}{a'}
&\le 2 \sum_{a''df_1 \in \sP(\lambda_0,\lambda_{\ell})}\frac{\rho(a''df_1)}{a''df_1}
\sum_{f_1<f_2\le e^{\eta}f_1}\frac{\rho(f_2)}{f_2}\\
&\le 2 g^{k_\ell}\sum_{a''df_1 \in \sP(\lambda_0,\lambda_{\ell})}
\frac{\rho(a''df_1)}{a''df_1}
\sum_{f_1<f_2\le e^{\eta}f_1}\frac{1}{f_2}\\
&\le 4g^{k_\ell}\eta \sum_{a''df_1 \in \sP(\lambda_0,\lambda_{\ell})}\frac{\rho(a''df_1)}{a''df_1}\\
&\le 4g^{k_\ell}\eta \prod_{\lambda_0 <p\le \lambda_{\ell}}\Big(1+\frac{\rho(p)}{p}\Big)^3\\
&\le 4g^{k_\ell} \eta \exp\Big\{ 3\sum_{\lambda_0<p\le \lambda_{\ell}}\frac{\rho(p)}{p}\Big\}\\ 
&\le  g^{k_\ell} 2^{3\ell+2}\eta,
\end{align*}
where we used \eqref{partion-merten} in the last step.
Inserting this last bound into \eqref{case-ii}, we see that
the contribution to the right side of \eqref{W-upperbound}
from the sets $Y,Z$ in case (ii) is at most
\[
\frac{(2\log 2)^k}{b_m!\cdots b_h!} V, \qquad
V = g^{k_\ell} 2^{3\ell+2} b_m!\cdots b_{\ell}! \eta.
\]
By assumption, $k_\ell \le 4\cdot 2^{\ell/2}$.  Using \eqref{eta-caseii} and the bound $b_j \le 2^{j/2}$, we see that
\dalign{
V \le g^{2^{\ell/2+2}} 2^{3\ell+2} (2^{\lceil \ell/2 \rceil}!)^\ell \exp \{-2^{\ell-1-c_5}\} \le 0.01
}
if $M$ is large enough, depending on $F$ (recall that $\ell\ge m\ge M$).

Combining the contributions from the case $Y=Z$ and $Y\ne Z$,
we immediately get the required result.
\end{proof}

We now stitch together the contribution from many  sets
$\cA(\mathbf{b})$, analogous to Lemma 4.8 in \cite{Ford08-1}.
The proof is nearly identical, and so we only sketch it,
indicating changes from \cite{Ford08-1}.

\begin{lemma}\label{F-lem48analog}
Suppose $y\ge y_0=y_0(F,\delta,C)$ and $0<\eta\le 2^{-M-1}\log z$. Define
\begin{align}
&v=\Big\lfloor \frac{\log\log (z)-\max(0, \log \eta)}{\log 2}-M+1\Big\rfloor, \label{vdef}\\
&s=M+\max\Big( 0, \Big\lfloor \frac{\log \eta}{\log 2}\Big\rfloor\Big)-\frac{\log (c_6 g\eta)}{\log 2}-c_5 - 5.\label{sdef}
\end{align}
If 
\be\label{k-range}
\max(10M,v) \le k\le \min \big( v+s - M/3-1, 100(v-1) \big),
\ee
then
$$
\ssum{a\in \sP(\max(D,e^{\eta}),z) \\ \omega(a)=k}\frac{(2\tau(a)-W(a; \eta))\rho(a)}{a}
\gg \frac{(2v\log 2)^k (k-v+1)}{(k+1)!}.
$$
\end{lemma}

\begin{proof}
Let $m=M+\max\Big(0, \lfloor \frac{\log \eta}{\log 2}\rfloor\Big)$, put $h=v+m-1$ and define $\sB$ to be the
set of vectors $\mathbf{b}=(b_m,\ldots,b_h)$ satisfying
\begin{enumerate}
\item[(a)] $b_j=0\ (1\le j\le m-1)$;
\item[(b)] $b_m+\cdots +b_h=k$; and
\item[(c)] $b_j \le M+(j+1-M)^2$ for all $j\ge m$.
\end{enumerate}
We  assume that $M \ge c_5+1$, which ensures,
by \eqref{eq-rangeoflambda}, that
$P^-(a)>\lambda_{m-1}>e^{\eta}$
whenever $a\in \sA(\bb)$ and $\bb\in \sB$.
We also have $h\le \frac{\log\log z}{\log 2}-c_5$, and thus for such $a$ we have also
$P^+(a) \le \lambda_h \le z$.
That is, $$\bigcup_{\bb\in \sB} \sA(\bb) \subset \sP(\max(D,e^\eta),z).$$
 By the definition of the sets
$E_j$, for any $\bb\in \sB$, we have
\begin{align}\label{Feq10.8}
\sum_{a\in \mathscr{A}({\bf b})}\frac{\rho(a)}{a}
&=\prod_{j=m}^h \frac{1}{b_j!}\Big(\sum_{p_1\in E_j}\frac{\rho(p_1)}{p_1}\ssum{p_2\in E_j\\ p_2\ne p_1}\frac{\rho(p_2)}{p_2}\cdots \ssum{p_{b_j}\in E_j\\ p_{b_j}\not\in \{ p_1,...,p_{b_j-1}\}}\frac{\rho(p_{b_j})}{p_{b_j}}\Big)\\
\nonumber&\ge \prod_{j=m}^h \frac{1}{b_j!}\Big( \sum_{p\in E_j}\frac{\rho(p)}{p}-\frac{b_j-1}{\lambda_{j-1}}\Big)^{b_j}\\
\nonumber &\ge \prod_{j=m}^h \frac{1}{b_j!}\Big( \log 2-\frac{b_j}{\lambda_{j-1}}\Big)^{b_j}\\
\nonumber &\ge\frac{(\log 2)^k}{b_m!\cdots b_h!}\prod_{j=m}^h\Big( 1-\frac{2^{j/10}}{\exp\{2^{j-1+c_3-c_4}\}}\Big)^{2^{j/10}}\\
\nonumber &\ge 0.999\frac{(\log 2)^k}{b_m!\cdots b_h!}
\end{align}
provided $M$ is large enough (recall $j\ge m \ge M$).
Combining this with Lemma \ref{lemWupperbound} and (\ref{Feq10.8}),
we see that
\be\label{Lem48-1}
\sum_{a\in \mathscr{A}({\bf b})}\frac{(2\tau(a)-W(a;\eta))\rho(a)}{a}\ge \frac{(2\log 2)^k}{b_m!\cdots b_h!}\Big[ 0.9- (2^{c_5}c_6 g)\eta\sum_{j=m}^h 2^{-j+b_m+\cdots +b_j}\Big].
\ee
For $1\le i\le v$, set $g_i=b_{m-1+i}$. 
Let $\sG$ denote the set of vectors $\mathbf{g}=(g_1,\ldots,g_v)$
such that
\begin{enumerate}
\item[(d)] $g_1+\cdots + g_v = k$;
\item[(e)] $g_i \le M + i^2$ for all $i$;
\item[(f)]
$\displaystyle 2^{m-1}\sum_{j=m}^h 2^{-j+b_m+\cdots+b_j}=\sum_{i=1}^v 2^{-i+g_1+\cdots+g_i}\le 2^{s+1}.$
\end{enumerate}

Clearly, (d) implies (b).
Since $m\ge M$, item (e) implies (c).
That is, $\mathbf{g} \in \sG$ implies that $\bb\in \sB$.
From the definition of $s$ and the inequality in (f),
we have $(2^{c_5}c_6  g)\eta 2^{s-m+2} \le 2^{-3}$.  By \eqref{Lem48-1}
and the equality in (f), we conclude that for all
$\mathbf{g}\in \sG$, and with $b_j=g_j-m+1$ for each $j\ge m$,
\[
\sum_{a\in \mathscr{A}({\bf b})}\frac{(2\tau(a)-W(a; \eta))\rho(a)}{a}\ge \frac{(2\log 2)^k}{2g_1!\cdots g_v!}.
\]
The argument on p. 418--419 of \cite{Ford08-1} then
shows that
\be\label{sumgi}\ssum{a\in \sP(\max(D,e^{\eta}),z) \\ \omega(a)=k}\frac{(2\tau(a)-W(a; \eta))\rho(a)}{a}
\ge (2\log 2)^k
\sum_{\mathbf{g}\in \sG} \frac{1}{g_1!\cdots g_v!} \ge (2v\log 2)^k
\text{Vol}\,  \Gamma_k(s,v),
\ee
where $\Gamma_k(s, v)$ is the set of ${\bf\xi}=(\xi_1,..., \xi_k)\in \mathbb{R}^k$ satisfying

{\rm (i)} $0\le \xi_1\le \cdots\le \xi_k<1$;

{\rm (ii)} For $1\le i\le \sqrt{k-M}$, $\xi_{M+i^2}>i/v$
and $\xi_{k+1-(M+i^2)}<1-i/v$;

{\rm (iii)} $\sum_{j=1}^k 2^{j-v\xi_j}\le 2^s$.

  We note that our condition (e)
is weaker than the corresponding condition in \cite{Ford08-1},
thus the sum on the left side of \eqref{sumgi} is greater than the sum considered in 
\cite{Ford08-1}.  We easily verify that, if $M$ is sufficiently 
large, then $s\ge M/2+1$.  Thus, by \eqref{k-range}, all of the
hypotheses of \cite[Lemma 4.9]{Ford08-1} are satisfied,
and we conclude that
$$
{\rm Vol}(\Gamma_k(s, v))\gg \frac{k-v+1}{(k+1)!}.
$$
Inserting this into \eqref{sumgi}, this completes the proof.
\end{proof}

{\bf Proof of the lower bounds in Theorem \ref{thm:main}.}
Suppose $2\le y\le x^{1-\delta}$ and $\frac{1}{\log^C y}\le \eta\le 1/100$, and define $\beta, \xi$ by (\ref{etabetaxi})
and $G(\beta)$ by \eqref{Gbeta}.
Let $y\ge y_0(F,C,\delta)$.
Define $v$ and $s$ by \eqref{vdef}, \eqref{sdef}, respectively.
We will apply Lemma \ref{F-lem48analog} for all $k$ satisfying
\be\label{krange2}
\bigg(1+\frac{\beta}{10B}\bigg) v\le k\le \min(1+\beta, \log 4)v.
\ee
This includes at least one value of $k$ since $\frac{\log 100}{\log\log y}\le \beta\le B$. Also, by (\ref{etabetaxi}),
$$k-v\le \beta v=\frac{-\log \eta}{\log\log y}v\le s-M/3-1,$$
and we have that $v \ge 10M$ for large enough $y_0$.
Hence, \eqref{k-range} holds for all $k$ satisfying 
\eqref{krange2}.
For each such $k$ in \eqref{krange2}, we obtain
$$
\ssum{a\in \sP(\max(D,e^\eta),z) \\ \omega(a)=k}\frac{(2\tau(a)-W(a; \eta))\rho(a)}{a}
\gg \beta\frac{(2v\log 2)^k}{k!}
$$
Applying Lemma \ref{FQlower-W} and using Lemma \ref{Norton} to bound the resulting sum over $k$ (cf., p. 397--398 in \cite{Ford08-1}), we see that
$$
H_F(x, y, z)-H_F(x/2,y,z) \gg \frac{\beta \eta (1+\eta) x}{\log^2 y} \sum_{k:\eqref{krange2}} \frac{(2v\log 2)^k}{k!}
\gg \frac{\beta x}{\max(1, -\xi)(\log y)^{G(\beta)}}.
$$
This gives the lower bound in Theorem \ref{thm:main} 
when $\eta \le \frac{1}{100}$.

Next, let 
$\gamma=2^{-M-c_5}(c_6 g)^{-1}\delta$, which is smaller than $\delta/3$,
 and suppose that $\frac{1}{100}\le \eta\le \gamma\log y$.
Apply Proposition \ref{FQ-HLP}, followed by Lemma \ref{F-lem48analog}
with the single term $k=v$.
Recalling that $\eta = u\log y$, we conclude that
\dalign{
H_F(x,y,z)-H_F(x/2,y,z) \gg \frac{\eta^2 x}{\log^2 y} \, \frac{(2v\log 2)^v}{(v+1)!} \gg \frac{x u^{\cE}}{(\log \frac2{u})^{3/2}},
}
as required for Theorem \ref{thm:main}. \qed

Finally, when $y^{1+\gamma} \le z\le y^2$ we have trivially
\[
H_F(x, y, z)-H_F(x/2,y,z) \ge H_F(x,y,y^{1+\gamma})-H_F(x/2,y,y^{1+\gamma}) \gg x.
\]

%%%%%%%%%%%%%%%%%%%%%%%%%%%%%%%%%%%%%%%%%%%
%
%
%
\section{The upper bound in Theorem \ref{thm:main}, Part I\label{sec:upper}}
%
%
%%%%%%%%%%%%%%%%%%%%%%%%%%%%%%%%%%%%%%%%%%%%

In this section, we establish the
principal local-to-global result needed for the upper bound
in Theorem \ref{thm:main}.
A crucial tool from \cite{Ford08-1} is, however, unavailable
because
 if $g=\deg(F) \ge 2$, $n \order x$ and
$d|F(n)$ with $y<d\le z\le x^{1-\delta}$, then the complementary divisor $F(n)/d$ is  $\gg x^{g-1+\delta}$ and this
is too large to handle.  We get around this with another
method (surrounding the parameters $A_{n,d}$, $B_{n,d}$ below).
Recall that $\sP(s,t)$ is the set of square-free integers,
all of whose prime factors lie in $(s,t]$.

\begin{proposition}\label{upper-sumL}
Let $C$ and $\delta$ be two positive real numbers with $0<\delta<1<C$.
Suppose $y_0=y_0(F,\delta,C)$ is sufficiently large.  Then for
$y_0 \le y<z=e^{\eta}y\le x^{1-\delta}$ and 
$\frac{1}{\log^C y}\le \eta $, we have
$$
H_F(x, y, z)\ll \frac{x}{\log^2 y}\ssum{a\in \sP(D,z)}\frac{L(a; \eta)\rho(a)}{\varphi_F(a)},
$$
where $D$ is defined in \eqref{Q}.
\end{proposition}

\subsection{Reduction of complicated sums to simpler ones.}

In this subsection, we present ways of bounding
certain complicated sums by simpler ones.
Our main result is similar in spirit to Lemma 3.3 of \cite{Ko}.
For all positive integers $n$ with $n>\sqrt{X}$, we define
\be\label{hdef}
h(n;X):=\min\Big\{ \text{prime}\ q: \prod_{p^\nu\| n, p\le q}p^\nu >\sqrt{X}\Big\}.
\ee

\begin{lemma}\label{sums-Sh}
Let $100\le X\le z$.  Then
\be\label{sum-h}
\ssum{\ell > \sqrt{X} \\ P^-(\ell)>D \\ P^+(\ell)\le z} \frac{L(\ell;\eta)\rho(\ell)}{\varphi_F(\ell)
\log^2 h(\ell;X)} \ll \frac{1}{\log^2 X}\pfrac{\log z}{\log X}^{4g}\sum_{a\in \sP(D,z)}\frac{L(a;\eta)\rho(a)}{\varphi_F(a)}.
\ee
In addition,
\be\label{sum-P+}
\ssum{P^-(\ell)>D \\ P^+(\ell)\le z} \frac{L(\ell;\eta)\rho(\ell)}{\varphi_F(\ell)
\log^2 (P^+(\ell) + z^{3/4}/\ell)} \ll \frac{1}{\log^2 z} \sum_{a\in \sP(D,z)}\frac{L(a;\eta)\rho(a)}{\varphi_F(a)}.
\ee
\end{lemma}

\begin{proof}
Our first goal is to prove that
\be\label{goal-1}
\ssum{\ell > \sqrt{X} \\ P^-(\ell)>D \\ P^+(\ell)\le z} \frac{L(\ell;\eta)\rho(\ell)}{\varphi_F(\ell)
\log^2 h(\ell;X)} \ll \frac{1}{\log^2 X}\pfrac{\log z}{\log X}^{4g} \ssum{P^-(a)>D \\ P^+(a)\le z}\frac{L(a;\eta)\rho(a)}{\varphi_F(a)}.
\ee
Let $f(n)=\frac{L(n;\eta)\rho(n)}{\varphi_F(n)/n}$
for $(n,Q)=1$ and $f(n)=0$ for $(n,Q)>1$.  By Lemma \ref{Ford-lemma31} (i)
and the fact that $D>2g$ from \eqref{Q}, we see that 
$f(p^\nu m) \le (4g)^\nu f(m)$ whenever $p$ is prime, $(m,p)=1$ and $\nu \ge 1$.
Also by \eqref{Q}, we have $D \ge 10g$.
First,  the part of the sum on the left side of \eqref{goal-1} corresponding to those $\ell$
with $h(\ell;X) > \sqrt{X}$ has the desired upper bound.
Now consider the case $h(\ell;X)\le \sqrt{X}$.
Let $H$ be the unique real number satisfying
\[
H^{1/2}<h(\ell;X)\le H,\quad \ H=(4g)^{2^k}\ \text{ for some non-negative integer }k.
\]
Fix $H$ and consider the numbers $\ell$ corresponding to $H$. 
Decompose each $\ell$ uniquely as
$$
\ell=\ell_1 \ell_2, \ P^+(\ell_1)\le H<P^-(\ell_2).
$$
By the definition \eqref{hdef} of $h()$, $\ell_1>\sqrt{X}$.  We also have
$$
f(\ell)\le f(\ell_1)(4g)^{\Omega(\ell_2)}.
$$
Taking $\kappa=4g+4$, and encode the condition $\ell_1>\sqrt{X}$
by introducing a factor $\pfrac{\log \ell_1}{\log X^{1/2}}^{\kappa}$.
Since $H \ge 4g$,
\begin{align*}
\ssum{\ell > \sqrt{X} \\ P^+(\ell)\le z\\ H^{1/2}<h(\ell;X)\le H}\frac{f(\ell)}{\ell\log^2 h(\ell;X)}
&\ll \frac{1}{\log^2 H}\ssum{\ell_1>\sqrt{X}\\ P^+(\ell_1)\le H}\frac{f(\ell_1)}{\ell_1}\ssum{P^+(\ell_2)\le z\\ P^-(\ell_2)>H}\frac{(4g)^{\Omega(\ell_2)}}{\ell_2}\\
&\ll  \frac{1}{\log^2 H}\ssum{P^+(\ell_1)\le H}\frac{f(\ell_1)}{\ell_1}\Big(\frac{\log \ell_1}{\log X^{1/2}}\Big)^{\kappa}\Big(\frac{\log z}{\log H}\Big)^{4g}\\
&= \frac{(\log z)^{4g}}{(\log H)^{4g+2}(\log X^{1/2})^{\kappa}}
\ssum{P^+(a)\le H}\frac{f(a)\log^{\kappa} a}{a}.
\end{align*}
For the final sum on the right side,
the argument in Lemma 3.3 in \cite{Ford08-2b} or Lemma 2.2 in \cite{Ko} gives
\begin{align*}
\ssum{P^+(a)\le H\\ P^-(a)>D}\frac{f(a)\log^{\kappa}a}{a}
\ll (4g)^{\kappa}(\log H)^{\kappa}\ssum{P^+(b)\le H\\ P^-(b)>D}\frac{f(b)}{b}.
\end{align*}
Finally, sum over $H$ and recall that $X\le z$.  We obtain
\begin{align*}
\ssum{\ell >X^{1/2} \\ P^+(\ell)\le z \\ h(\ell;X)\le X^{1/2}}\frac{f(\ell)}{\ell\log^2 h(\ell;X)}
&=\ssum{k: X^{1/2^k} \ge 2} \ssum{X^{1/2}<\ell\le zX\\ P^+(\ell)\le z\\ X^{1/2^{k+1}}<h(\ell;X)\le X^{1/2^k}}\frac{f(\ell)}{\ell\log^2 h(\ell;X)}\\
&\ll \ssum{k=1}^{\infty}\frac{(\log z)^{4g}}{(\log X^{1/2^k})^{4g+2}(\log X^{1/2})^{\kappa}}
(\log X^{1/2^k})^{\kappa}\ssum{P^+(b)\le X^{1/2^k}}\frac{f(b)}{b}\\
&\ll \frac{1}{\log^2 X}\pfrac{\log z}{\log X}^{4g}\Big(\ssum{k=1}^{\infty}2^{-k(\kappa-4g-2)}\ssum{P^+(b)\le X}\frac{f(b)}{b}\Big)\\
&\ll \frac{1}{\log^2 X}\pfrac{\log z}{\log X}^{4g}\ssum{P^+(b)\le z}\frac{f(b)}{b}.
\end{align*}
This completes the proof of \eqref{goal-1}.

Next, we remove the squarefull part of $a$ from the sum.  
Each $a\in \NN$ may be uniquely decomposed as $a=a_1a_2$, where
$(a_1,a_2)=1$, $a_1$ is squarefree and $a_2$ is squarefull.
As $\rho$ and $\phi_F$ are multiplicative, 
$L(a;\eta) \le \tau(a_2) L(a_1;\eta)$ by Lemma \ref{Ford-lemma31} (i).
Recalling  \eqref{rho-upper} and \eqref{Q-phi}, we see  that
\begin{align}
\ssum{P^-(a)>D \\ P^+(a)\le z}\frac{L(a;\eta)\rho(a)}{\varphi_F(a)} &\le 
\sum_{a_1\in \sP(D,z)} \frac{L(a_1;\eta)\rho(a_1)}{\varphi_F(a_1)} 
\sum_{P^-(a_2)>D} \frac{\tau(a_2)\rho(a_2)}{\varphi_F(a_2)} \notag \\
&\le \sum_{a_1\in \sP(D,z)} \frac{L(a_1;\eta)\rho(a_1)}{\varphi_F(a_1)} 
\prod_{p>D} \(1+\frac{6g}{p^2} + \frac{8g}{p^3}+\cdots\) \notag \\
&\ll\sum_{a_1\in \sP(D,z)} \frac{L(a_1;\eta)\rho(a_1)}{\varphi_F(a_1)}.
\label{sqfree-removal}
\end{align}
This proves \eqref{sum-h}.  

Next, break the sum on the left side of \eqref{sum-P+} into two parts,
corresponding to $a\le z^{1/2}$ and to $a>z^{1/2}$.  In the first part,
$\log^2(P^+(a) + z^{3/4}/a) \gg \log^2 z$, and the desired bound follows
from \eqref{sqfree-removal}.  Since $H(\ell;X)\le P^+(\ell)$, the second part is majorized by
the left side of \eqref{sum-h} with $X=z$, and thus 
we see that \eqref{sum-P+} follows from \eqref{sum-h}.
\end{proof}

%%%%%%%%%%%%%%%%%%%%%%%%%%%%%%%%%%%%%%%%%%%%%%%%%%%%%%%%%
%
\subsection{Proof of Proposition \ref{upper-sumL}.} 
%
%%%%%%%%%%%%%%%%%%%%%%%%%%%%%%%%%%%%%%%%%%%%%%%%%%%%%%%%%%%%%%%
Let $\mathcal{A}$ be the set of positive integers $n\le x$ satisfying

(i) $\tau(F(n); y, z)\ge 1$;

(ii) $n>\frac{x}{(\log y)^{C+2}}$;

(iii) if $p$ is prime with $p|F(n)$ and $(\log y)^{C+2}<p\le z$, then $p^2\nmid F(n)$;

(iv) $\ds \sprod{p^\nu \|F(n)\\ p\le (\log y)^{C+2}}p^\nu\le \exp\{(\log\log z)^3\}$.\\

The number of integers $n\le x$ not satisfying (iii)
is at most
\begin{align*}
\ssum{(\log y)^{C+2}<p\le z}\Big( \frac{x\rho(p^2)}{p^2}+\rho(p^2)\Big)
\ll x\ssum{p>(\log y)^{C+2}}\frac{1}{p^2}+z\ll \frac{x}{(\log y)^{C+2}}.
\end{align*}
By Lemma \ref{Ten-Lem37},
the number of integers $n\le x$ failing (iv) is
\begin{align*}
\#\{n\le x: \sprod{p^\nu ||F(n)\\ p\le (\log y)^{C+2}}p^\nu> \exp\{(\log\log z)^3\}\}
\ll x\exp\Big\{-c_4\frac{(\log\log z)^3}{\log ((\log y)^{C+2})}\Big\}\ll \frac{x}{(\log y)^{C+2}}.
\end{align*}
So we have
\be\label{HFA}
H_F(x, y, z)\le |\mathcal{A}|+O\bigg(\frac{x}{(\log y)^{C+2}}\bigg).
\ee
Each integer $d\in (y, z]$ has a unique decomposition
\begin{align}\label{d0d1-definition}
d=d_0d_1, \quad P^+(d_0) \le D < P^-(d_1).
\end{align}
If $d\in (y, z]$ and $d|F(n)$, then by (iv), we have $P^+(d)>(\log y)^{C+2}$ since $z^{1/2}\le y<d$.  It follows that $d_1>1$.  Also, by (iv),
\be\label{d0}
d_0 \le y^{1/10}.
\ee
Let
\begin{align}\label{X}
X:=\min\{z, x^{\delta/2}\}.
\end{align}
In particular, $100 \le X \le z$, as required for Lemma
\ref{sums-Sh}.
For each $d\in (y, z]$  with $\rho(d)>0$, let
$$
\mathcal{A}_d:=\{n\in \mathcal{A}: d|F(n)\}.
$$
For each $d$ and $n\in \mathcal{A}_d$, by (iii) and (iv) $F(n)$ has a unique decomposition in the form
\begin{align}\label{decomF(n)}
F(n)=Q_{n,d}M_{n,d}A_{n,d}B_{n,d},
\end{align}
with the conditions
\begin{align}\label{DnMnAnBn-definition}
Q_{n,d}|Q^\infty, \; M_{n,d}|d_1^\infty, \; (A_{n,d}B_{n,d}, Q d_1)=1\; \text{and}\ \; P^+(A_{n,d})<P^-(B_{n,d}),
\end{align}
where we choose $A_{n,d}$ as large as possible such that
\begin{align}\label{An-restriction}
A_{n,d}\le X\ \text{and}\ P^+(A_{n,d})< P^+(d).
\end{align}
In particular, $d_0|Q_{n,d}$ and $d_1 | M_{n,d}$.
%
% KF: it appears that we do not need this.
%
%We claim that
%$$B_{n,d}>1$$
%in the decomposition (\ref{decomF(n)}) for each $n\in \mathcal{A}_d$. Otherwise, we  have
%$$F(n)=E_nM_nA_{n,d}\le E_nM_n X=dXe,$$
%where $e|(dD)^{\infty}$.
%Since $d\le z$ and $D=O(1)$, we obtain by (iii) that $P^+(e)\le (\log y)^{C+2}$. It then follows from (iv) that
%$e\le \exp\{(\log\log z)^3\}$,
%which, together with \eqref{X}, implies that $F(n)\le zX\exp\{(\log\log z)^3\}\ll x^{1-\delta/3}$. So $n\ll x^{(1-\delta/3)/g}$, which
%contradicts (ii). Hence $B_{n,d}>1$ for each $n\in \mathcal{A}_d$.

Now fix $d$ and suppose that $n\in \mathcal{A}_d$.
Define
\[  %\begin{align}\label{largestprimefactorofd}
p=P^+(d).
\]
Then by (iv), 
\[  %\label{pgreaterthanlog10z}
p=P^+(d_1)\ \text{and}\ p>(\log y)^{C+2}.
\]

 Write
\begin{align}\label{ell-definition}
\ell =(d_1/p)A_{n,d},
\end{align}
where $d_1$ and $A_{n,d}$ are defined as in (\ref{d0d1-definition})
and (\ref{decomF(n)}) under the constraints (\ref{DnMnAnBn-definition}) and (\ref{An-restriction}).
Thus we derive from (iii), (iv) and (\ref{An-restriction}) that $P^+(\ell)<p$. Moreover, it is easy to see from (\ref{d0d1-definition}) and (\ref{ell-definition}) that $\tau(pd_0\ell; y, z)\ge 1$,
which implies
$$\log (y/p)\in \mathscr{L}(d_0\ell; \eta).$$
Next, set
\begin{align}\label{theta-definition}
\vartheta = \frac{\log X}{2\log z},
\end{align}
so that by \eqref{X} we have
\be\label{theta-bounds}
1 \ll \vartheta \le \frac12.
\ee
Partition the set $\mathcal{A}_d$
into the disjoint sets
\begin{align*}
\mathcal{A}_{d, 1}:=\{n\in \mathcal{A}_d: P^-(B_{n,d})>p^{\vartheta}\}, \qquad
\mathcal{A}_{d, 2}:=\{n\in \mathcal{A}_d: P^-(B_{n,d})\le p^{\vartheta}\}.
\end{align*}

First, we consider the set $\bigcup_{d\in (y, z]}\mathcal{A}_{d, 1}$.   By \eqref{X},  we have
$$p^{\vartheta}\le z^{1/2} \le x^{1/2},
 \qquad pd_0\ell=d_0d_1A_{n,d}\le zX<x^{1-\delta/2}$$
and
$$P^+(pd_0\ell)=p = (p^{\vartheta})^{1/\vartheta}.$$
Given $p,\ell$ and $d_0$ with $p\ell d_0 | F(n)$, 
it follows that all prime factors of $F(n)$ either
divide $p\ell Q$ or are greater than $p^\vartheta$.
By Lemma \ref{T90-1-L34} with $K'=1/\vartheta$ and $K=\delta/2$, together with \eqref{theta-bounds}, we obtain
\begin{align*}
\Big| \bigcup_{d\in (y, z]}\mathcal{A}_{d, 1}\Big|
&\le \ssum{d_0\ell\le zX\\ P^+(\ell)\le z}\ssum{P^+(\ell)<p\le z\\p\ge y/\ell d_0\\ \log(y/p)\in \mathscr{L}(d_0\ell; \eta)}\ssum{n\le x,  pd_0\ell|F(n)\\ q|F(n)\Rightarrow q|p\ell Q\ \text{or}\ q>p^{\vartheta}}1\\
&\ll \ssum{d_0\ell\le zX\\ P^+(\ell)\le z}\ssum{P^+(\ell)<p\le z\\p\ge y/\ell d_0\\ \log(y/p)\in \mathscr{L}(d_0\ell; \eta)}\frac{x}{\log p}\frac{\rho(d_0)}{d_0}\frac{\rho(\ell)}{\varphi_F(\ell)}
\frac{\rho(p)}{p}.
\end{align*}
Since $\mathscr{L}(a; \eta)$ is the disjoint union of intervals
of length between $\eta/2$ and $\eta$, repeated use of Lemma \ref{Mertenformula} gives
\begin{align*}
\ssum{P^+(\ell)<p\le z\\p\ge y/\ell d_0\\ \log(y/p)\in \mathscr{L}(d_0\ell; \eta)}\frac{\rho(p)}{p\log p}
\ll \frac{L(d_0\ell; \eta)}{\log^2 (y/d_0\ell+P^+(\ell))}.
\end{align*}
Using Lemma \ref{Ford-lemma31} (i), $L(d_0\ell;\eta) \le L(\ell;\eta)\tau(d_0)$.
Hence, by \eqref{d0},
\[
\Big| \bigcup_{d\in (y, z]}\mathcal{A}_{d, 1}\Big| \ll x \ssum{P^+(\ell)\le z}\frac{L(\ell; \eta)\rho(\ell)}{\varphi_F(\ell)
\log^2(y^{9/10}/\ell+P^+(\ell))} \sum_{d_0|Q^\infty}
\frac{\tau(d_0)\rho(d_0)}{\varphi_F(d_0)}.
\]
By \eqref{rho-upper},
\be\label{sum-d0}
\sum_{d_0|Q^\infty} \frac{\rho(d_0)\tau(d_0)}{d_0} = \prod_{p\le D}\Big(\ssum{\nu=0}^\infty
\frac{\tau(p^\nu)\rho(p^\nu)}{p^\nu} \Big) \ll 1.
\ee
Using Lemma \ref{sums-Sh} part \eqref{sum-P+} and recalling \eqref{X}, we have
\be\label{Ad1}
\Big| \bigcup_{d\in (y, z]}\mathcal{A}_{d, 1}\Big|\ll x \ssum{d_0\ell\le zX\\ P^+(\ell)\le z}\frac{L(\ell; \eta)\tau(d_0)\rho(\ell)}{\varphi_F(\ell)
\log^2(y/d_0\ell+P^+(\ell))}
\ll \frac{x}{\log^2 y}\ssum{a\in \sP(D,z)}\frac{L(a; \eta)\rho(a)}{\varphi_F(a)}.
\ee

Next, we estimate the size of $\bigcup_{d\in (y, z]}\mathcal{A}_{d, 2}$. If $P^-(B_{n,d})\le p^{\vartheta}$, then by the definition of
$A_{n,d}$, we obtain $A_{n,d}P^-(B_{n,d})>X.$
 Hence by \eqref{ell-definition} and \eqref{theta-definition}, we have
\begin{align}\label{AngreatersqrtX}
\ell\ge A_{n,d}>X/P^-(B_{n,d}) \ge Xp^{-\vartheta}\ge Xz^{-\vartheta} \ge X^{1/2}.
\end{align}
Recalling the definition \ref{hdef} of $h(\cdot)$, we see that
$h(\ell;X)$ and $h(A_{n,d};X)$ are well-defined.
Then by (\ref{ell-definition}) and (\ref{AngreatersqrtX}), we have
\[  %\begin{align}\label{h(l)lessthanp}
h(\ell;X)\le h(A_{n,d};X)\le P^+(A_{n,d})< P^-(B_{n,d})\le p^{\vartheta}<p,
\]
where we invoked \eqref{theta-bounds} in the last step.
Hence, Lemma  \ref{T90-1-L34} implies (in the sums, $d_0|Q^\infty$ and $(\ell,Q)=1$) 
\begin{align*}
\bigg|\bigcup_{d\in (y, z]}\mathcal{A}_{d, 2}\bigg|
&\le \ssum{d_0\ell\le zX\\ \ell>X^{1/2}\\ P^+(\ell)\le z} \;\;
\ssum{P^+(\ell)<p\le z\\ p\ge y/d_0\ell\\ \log(y/p)\in \mathscr{L}(d_0\ell)}
\ssum{n\le x, pd_0\ell|F(n)\\ q|F(n)\Rightarrow q|p\ell Q\ \text{or}\ q>h(\ell;X)}1\\
&\ll \ssum{d_0\ell\le zX\\ \ell>X^{1/2}\\ P^+(\ell)\le z\;\;}\ssum{P^+(\ell)<p\le z\\p\ge y/d_0\ell\\ \log(y/p)\in \mathscr{L}(d_0\ell)}\frac{x}{\log h(\ell;X)}\frac{\rho(d_0)}{d_0}\frac{\rho(\ell)}{\varphi_F(\ell)}
\frac{\rho(p)}{p}.
\end{align*}

As above, applying Lemma \ref{Mertenformula} repeatedly, we obtain
\[
\ssum{P^+(\ell)<p\le z\\p\ge y/d_0\ell\\ \log(y/p)\in \mathscr{L}(d_0\ell)}\frac{\rho(p)}{p} \ll 
\frac{L(d_0\ell;\eta)}{\log(P^+(\ell)+\frac{y}{d_0\ell})}
\le \frac{L(d_0\ell;\eta)}{\log h(\ell;X)},
\]
since $h(\ell;X) \le P^+(\ell)$ (cf., the definition \eqref{hdef} of $h()$).  Thus,
\be\label{Ad2-1}
\bigg|\bigcup_{d\in (y, z]}\mathcal{A}_{d, 2}\bigg|
\ll x\ssum{d_0\ell\le zX\\ \ell>X^{1/2}\\ P^+(\ell)\le z}\frac{\rho(d_0)\rho(\ell)L(d_0\ell; \eta)}{d_0\varphi_F(\ell)\log^2 h(\ell;X)} \le x
\ssum{\ell>X^{1/2}\\ P^+(\ell)\le z} \frac{\rho(\ell)L(\ell; \eta)}{\varphi_F(\ell)\log^2 h(\ell;X)} \sum_{d_0|Q^\infty}  \frac{\rho(d_0)\tau(d_0)}{d_0},
\ee
where we used Lemma \ref{Ford-lemma31} (i) in the last step.
Thus, applying \eqref{sum-d0} and Lemma \ref{sums-Sh} part \eqref{sum-h} to the right side of \eqref{Ad2-1}, we find that
\be\label{Ad2}
\bigg|\bigcup_{d\in (y, z]}\mathcal{A}_{d, 2}\bigg|
\ll 
 \frac{x}{\log^2 z}\ssum{a\in \sP(D,z)}\frac{L(a; \eta)\rho(a)}{\varphi_F(a)}.
\ee
Finally, we combine \eqref{HFA}, \eqref{Ad1} and \eqref{Ad2} to
obtain
\[
H_F(x,y,z) \ll \frac{x}{\log^2 y} \ssum{a\in \sP(D,z)}\frac{L(a; \eta)\rho(a)}{\varphi_F(a)} + O\pfrac{x}{(\log y)^{C+2}}.
\]
The error term is negligible as
\[
 \ssum{a\in \sP(D,z)}\frac{L(a; \eta)\rho(a)}{\varphi_F(a)} 
 \ge L(1;\eta) = \eta \gg \frac{1}{(\log y)^C}.
\]
This completes the proof of Proposition \ref{upper-sumL}. 
 \qed

\section{The upper bound in Theorem \ref{thm:main}, Part II}

In this section, we complete the proof of the upper bound in
Theorem \ref{thm:main} using Proposition \ref{upper-sumL}.
This part of the argument follows \cite{Ford08-1} 
with only trivial modification.
Recall the partition of the primes larger than $D$
from Section 3, in particular \eqref{partion-merten}
and \eqref{eq-rangeoflambda}.
The following is analogous to Lemma 3.5 in \cite{Ford08-1}.

\begin{lemma}\label{Tkzupperbound}
Suppose $y\rightarrow \infty, z-y\rightarrow \infty$ and $0<\eta\le \log y$. Let
$$
v=\Big\lceil \frac{\log\log z}{\log 2}\Big\rceil
$$
and suppose $0\le k\le 10v$. Then
$$
T_k(z):=\ssum{a\in \mathscr{P}(D,z)\\\omega(a)=k}
\frac{L(a; \eta)\rho(a)}{\varphi_F(a)}\ll (\eta+1)(2v\log 2)^k U_k(v; \min(1, \eta)),
$$
where
$$U_k(v; t)=\int_{0\le \xi_1\le \cdots \le \xi_k\le 1}\min_{0\le j\le k}2^{-j}\big(2^{v\xi_1}+\cdots+2^{v\xi_j}+t\big)d{\boldsymbol \xi}.
$$
\end{lemma}
\begin{proof}
Consider $a=p_1\cdots p_k$ with $D<p_1<\cdots<p_k\le z$ and define $j_i$ by $p_i\in E_{j_i}\ (1\le i\le k)$.
Put $l_i=\frac{\log\log p_i}{\log 2}$. By Lemma \ref{Ford-lemma31} (ii) and (\ref{eq-rangeoflambda}),
$$
L(a; \eta)\le 2^k\min_{0\le i\le k}2^{-i}(2^{l_1}+\cdots+2^{l_i}+\eta)\le (\eta+1)2^{k+c_5}F({\bf j}),
$$
where
$$
F({\bf j})=\min_{0\le i\le k}2^{-i}(2^{j_1}+\cdots+2^{j_i}+\min(1, \eta)).
$$
Let $J$ denote the set of vectors ${\bf j}$ satisfying $0\le j_1\le \cdots\le j_k\le v+c_5-1$.
Then
\begin{align}\label{upperTk} 
T_k(z)\le (\eta+1)2^{k+c_5}\sum_{{\bf j}\in J}F({\bf j})\ssum{D<p_1<\cdots< p_k\\ p_i\in E_{j_i}\ (1\le i\le k)}\frac{\rho(p_1\cdots p_k)}{\varphi_F(p_1\cdots p_k)}.
\end{align}
If $b_j$ is the number of primes $p_i$ in $E_j$ for $1\le j\le v+c_5-1$, the sum over $p_1, ..., p_k$ above is at most
$$
\prod_{j=1}^{v+c_5-1}\frac{1}{b_j!} \Big(\sum_{p\in E_j}\frac{\rho(p)}{\varphi_F(p)}\Big)^{b_j}\le
((v+c_5)\log 2)^{k}\int_{R({\bf j})}1d{\bf \xi}
\le e^{10c_5}(v\log 2)^k\int_{R({\bf j})}1 d{\bf \xi},
$$
where
$$
R({\bf j})=\{ 0\le \xi_1\le \cdots\le \xi_k\le 1: j_i\le (v+c_5)\xi_i\le j_i+1\ \forall i\}\subseteq R_k.
$$
Finally, since $2^{j_i}\le 2^{(v+c_5)\xi_i}\le 2^{c_5}2^{v\xi_i}$
for each $i$,
\begin{align*}
\sum_{{\bf j}\in J}F(j)\int_{R({\bf j})}1d{\bf \xi}\le 2^{c_5}U_k(v; \min(1, \eta)).
\end{align*}
So by (\ref{upperTk}), we obtain
\[
T_k(z)\ll (\eta+1)(2v\log 2)^kU_k(v; \min(1, \eta)).\qedhere
\]
\end{proof}

When $ z_0(y) \le z \le y^{1+\delta/2}$, where 
$z_0(y)$ is defined in \eqref{z0},
the upper bound in Lemma \ref{Tkzupperbound} 
is identical to the bound in
\cite[Lemma 3.5]{Ford08-1} (taking $Q=1$ in this lemma).
Therefore, the proof of Lemma 3.7 in
\cite{Ford08-1} provides the required
upper bound for $\sum_k T_k(z)$.
Combined with Proposition \ref{upper-sumL} (replacing $\delta$ with $\delta/2$), this
gives the desired upper bound for $H_F(x,y,z)$
in Theorem \ref{thm:main}.
When $y+y/(\log y)^C \le z\le z_0(y)$, the upper bound
follows form the simple estimate
\[
H_F(x,y,z) \ll \sum_{y<d\le z} \frac{\rho(d) x}{d} \ll \eta x,
\]
a consequence of Lemma \ref{sumonrho(d)}.
Finally, when $z\ge y^{1+\delta/2}$, the trivial bound
$H_F(x,y,z)\le x$ suffices.

%%%%%%%%%%%%%%%%%%%%%%%%%%%%%
%%%%%%%%%%%%%%%%%%%%%%%%%%%%%
%%%%%%%%%%%%%%%%%%%%%%%%%%%%%
%%%%%%%%%%%%%%%%%%%%%%%%%%%%%

\end{document}